\documentclass[preprint,12pt,3p]{elsarticle}

\usepackage{amssymb, amsmath}
\usepackage{amsthm}
\usepackage{color}
\usepackage{enumitem}

\newtheorem{theorem}{Theorem}[section]

\theoremstyle{definition}
\newtheorem{definition}[theorem]{Definition}

\newtheorem{example}[theorem]{Example}

\theoremstyle{remark}
\newtheorem{remark}[theorem]{Remark}

\begin{document}

\begin{frontmatter}

\title{On forced oscillations of a constrained relativistic particle}

\author{Ivan Polekhin}


\ead{ivanpolekhin@gmail.com}



\begin{abstract}
We consider the forced motion of a relativistic particle constrained on a curve and present sufficient conditions for periodic oscillations by means of an illustrative geometrical approach. Obtained result is illustrated by a few examples including the forced relativistic pendulum.
\end{abstract}


\end{frontmatter}


\section{Introduction}
\label{sec1}

The problem considered in this paper originally stems from the work of G.\,Hamel who in 1922 proved \cite{hamel1922erzwungene} that the periodically forced classical pendulum have at least one periodic solution. After that, various authors considered not only classical forced pendulum, but also relativistic approximation of this problem. In particular, in~\cite{torres2008periodic} it was shown that for a forced relativistic pendulum with friction there always exists at least one periodic solution if the period of the external generalized force acting on the massive point is not too large. It is worth mentioning that for a classical forced pendulum with friction general behaviour of solutions is different: for a given period, there always exists an external generalized force such that there are no periodic solutions in the system~\cite{ortega2000non}. One can find a survey concerning forced oscillations in relativistic and classical systems in~\cite{mawhin2010periodic}. For a relativistic pendulum without friction it was also shown that there exists a periodic solution~\cite{brezis2010periodic} and, moreover, there are at least two different periodic solutions in this system~\cite{bereanu2012existence}. 

In the paper, we consider a generalization of the relativistic forced pendulum to the case of a periodically forced relativistic particle constrained on an arbitrary curve --- configuration space is assumed to be one-dimensional --- and present sufficient conditions for the existence of a periodic solution. In the general case, the governing equation of the considered system can be presented, in appropriate units, as follows
\begin{equation}
\label{maineq1}
\frac{d}{dt}\frac{\dot q}{\sqrt{1 - \dot q^2}} = f(t, q, \dot q),
\end{equation}
where $q \in \mathbb{R}/2\pi\mathbb{Z}$ and $f \in C^1(\mathbb{R}/T\mathbb{Z} \times \mathbb{R} \times \mathbb{R}, \mathbb{R})$, $T>0$. Note that equation~(\ref{maineq1}) is also considered in~\cite{bereanu2007existence} and \cite{bereanu2008periodic} in more general form. In particular, sufficient conditions for the existence of a periodic solution are presented.

Results~\cite{torres2008periodic},~\cite{mawhin2010periodic}, ~\cite{brezis2010periodic},~\cite{bereanu2007existence} and \cite{bereanu2008periodic} have been obtained by means of variational technique and Leray-Schauder degree arguments. We present different, purely topological, approach based on a result~\cite{srzednicki2005fixed} by R.\,Srzednicki, K.\,W{\'o}jcik, and P.\,Zgliczy{\'n}ski which allows one to study and prove periodicity in various systems by using finite-dimensional topological arguments and illustrative geometrical constructions. The last proves to be useful in our case since we consider low-dimensional systems.

The result is illustrated by a series of examples including forced relativistic pendulum and its modifications.

\section{Results}

\subsection{Supplementary definitions and results}
\label{sec2}
First, following~\cite{srzednicki2005fixed} we introduce some definitions slightly modified for our use. From now on, we assume that $v \in C^1(\mathbb{R}\times M, TM)$ is a  time-dependent vector-field on a smooth manifold $M$. 
\begin{definition}
For $t_0 \in \mathbb{R}$ and $x_0 \in M$, the map $t \mapsto x(t_0,x_0,t)$ is the solution for the initial value problem for the system $\dot x = v(t, x)$, such that $x(t_0,x_0,t_0)=x_0$.
\end{definition}
\begin{definition}
Let $W \subset \mathbb{R} \times M$. Define the {exit set} $W^-$ as follows. A point $(t_0,x_0)$ is in $W^-$ if there exists $\delta>0$ such that $(t+t_0, x(t_0,x_0,t+t_0)) \notin W$ for all $t \in (0,\delta)$.
\end{definition}
\begin{definition}
We call $W \subset \mathbb{R}\times M$ a {Wa\.{z}ewski block} for the system $\dot x = v(t,x)$ if $W$ and $W^-$ are compact.
\end{definition}
 Now introduce some notations. By $\pi_1$ and $\pi_2$ we denote the projections of $\mathbb{R}\times M$ onto $\mathbb{R}$ and $M$ respectively. If $Z \subset \mathbb{R}\times M$, $t\in\mathbb{R}$, then we denote
\begin{equation*}
Z_t=\{z \in M \colon (t,z) \in Z\}.
\end{equation*}
\begin{definition}
A set $W \subset [a,b] \times M$ is called a segment over $[a,b]$ if it is a block with respect to the system $\dot x = v(t,x)$ and the following conditions hold:
\begin{itemize}
\item there exists a compact subset $W^{--}$ of $W^-$ called the essential exit set such that
\begin{equation*}
W^-=W^{--}\cup(\{b\}\times W_b),\quad W^-\cap([a,b)\times M) \subset W^{--},
\end{equation*}
\item there exists a homeomorphism $h\colon [a,b]\times W_a \to W$ such that $\pi_1 \circ h = \pi_1$ and
\begin{equation}
\label{cond-2}
h([a,b]\times W_a^{--})=W^{--}.
\end{equation}
\end{itemize}
\end{definition}
\begin{definition}
Let $W$ be a segment over $[a,b]$. It is called periodic if
\begin{equation*}
(W_a,W_a^{--})=(W_b,W_b^{--}).
\end{equation*}
\end{definition}
\begin{definition}
For a periodic segment $W$, we define the corresponding monodromy map $m$ as follows
\begin{equation*}
m\colon W_a\to W_a, \quad m(x) = \pi_2 h(b,\pi_2 h^{-1}(a,x)).
\end{equation*}
\end{definition}
\begin{theorem} 
\label{main-th}(See
\cite{srzednicki2005fixed}) Let W be a periodic segment over $[a,b]$. Then the set
\begin{equation*}
U = \{ x_0 \in W_a \colon x(a,x_0,t) \in W_t\setminus W_t^{--}\,\mbox{for all}\,\, t \in [a,b] \}
\end{equation*}
is open in $W_a$ and the set of fixed points of the restriction $x(a,\cdot, b)|_U \colon U \to W_a$ is compact. Moreover, if $W$ and $W^{--}$ are ANRs then
\begin{equation*}
\mathrm{ind}(x(a,\cdot, b)|_U) = \Lambda(m) - \Lambda(m|_{W_a^{--}}).
\end{equation*}
Where by $\Lambda(m)$ and $\Lambda(m|_{W_a^{--}})$ we denote the Lefschetz number of $m$ and $m|_{W_a^{--}}$ respectively. In particular, if $\Lambda(m) - \Lambda(m|_{W_a^{--}}) \ne 0$ then $x(a,\cdot, b)|_U$ has a fixed point in $W_a$.
\end{theorem}

\subsection{Main result}

\begin{theorem}
\label{ourres}
If for some functions $h_1, h_2 \in C^2(\mathbb{R}/T\mathbb{Z},\mathbb{R})$, $h_1(t) < h_2(t)$, $|\dot h_1| < 1$, $|\dot h_2| < 1$ for all $t \in \mathbb{R}/T\mathbb{Z}$, the following conditions are satisfied
\begin{equation}
\label{maincond}
\begin{aligned}
&(1-\dot h_1^2)^{3/2}f(t,h_1,\dot h_1) - \ddot h_1 < 0,\\
&(1-\dot h_2^2)^{3/2}f(t,h_2,\dot h_2) - \ddot h_2 > 0,
\end{aligned}
\end{equation}
then there exists a periodic solution $q \colon \mathbb{R}/T\mathbb{Z} \to \mathbb{R}$ of (\ref{maineq1}). Moreover, $h_1(t) < q(t) < h_2(t)$ for all $t \in \mathbb{R}/T\mathbb{Z}$.
\end{theorem}

\begin{proof}
First, rewrite (\ref{maineq1}) as follows
\begin{equation}
\label{maineq2}
\begin{aligned}
&\dot q = p,\\
&\dot p = (1 - p^2)^{3/2}\cdot f(t,q,p).
\end{aligned}
\end{equation}
In order to be precise in our further consideration, in addition to the above system, we will assume that $\dot p = 0$ for $|p| > 1$. Systems~(\ref{maineq1}) and~(\ref{maineq2}) are equivalent provided $|p| < 1$. Consider the following subset of the extended phase space
\begin{equation*}
W = \{ t \in [0,T], q \in \mathbb{R}, p \in [-1, 1] \colon h_1(t) \leqslant q \leqslant h_2(t)\}.
\end{equation*}
Let us show now that for~(\ref{maineq2}) the essential exit set can be described as follows
\begin{equation*}
\begin{aligned}
W^{--} =&\{ t \in [0,T], q \in \mathbb{R}, p \in [-1, 1] \colon q = h_1(t), -1 \leqslant p \leqslant \dot h_1(t) \}\\
\cup&\{ t \in [0,T], q \in \mathbb{R}, p \in [-1, 1] \colon q = h_2(t), \dot h_2(t) \leqslant p \leqslant 1 \}.
\end{aligned}
\end{equation*}
Indeed, manifolds $p = 1$ and $p = -1$ are invariant for~(\ref{maineq2}) and, taking into account that $\dot q = p$, we have
\begin{equation*}
\begin{aligned}
\{ t \in [0,T], q &\in \mathbb{R}, p \in [-1, 1] \colon q = h_1(t), -1 \leqslant p < \dot h_1(t) \}\\
\cup&\{ t \in [0,T], q \in \mathbb{R}, p \in [-1, 1] \colon q = h_2(t), \dot h_2(t) < p \leqslant 1 \} \subset W^{--},\\
\{ t \in [0,T], q &\in \mathbb{R}, p \in [-1, 1] \colon q = h_1(t), \dot h_1(t) < p \leqslant 1 \}\\
\cup&\{ t \in [0,T], q \in \mathbb{R}, p \in [-1, 1] \colon q = h_2(t), -1 \leqslant p < \dot h_2(t) \} \not\subset W^{--}.
\end{aligned}
\end{equation*}
For any $t_0 \in [0, T]$, $q_0 = h_1(t_0)$, $p_0 = \dot h_1(t_0)$ we have
\begin{equation*}
\begin{aligned}
q(t_0,q_0,p_0,t) - h_1(t) = \frac{(t-t_0)^2}{2}((1-\dot h_1^2(t_0))^{3/2}f(t_0,h_1(t_0),\dot h_1(t_0)) - \ddot h_1(t_0)) + o((t-t_0)^2).
\end{aligned}
\end{equation*}
Therefore, from~(\ref{maincond}) we have that $(t_0, q_0, p_0) \in W^{--}$. By the same arguments, for any $t_0 \in [0, T]$, $q_0 = h_2(t_0)$, $p_0 = \dot h_2(t_0)$, we have  $(t_0, q_0, p_0) \in W^{--}$.

Now define a homeomorphism $h \colon [0, T] \times W_0 \to W$ as follows
\begin{equation*}
h \colon (t,q,p) \mapsto \left(t, \frac{h_1(t) - h_2(t)}{h_1(0) - h_2(0)}(q - h_1(0)) + h_1(t), L_1(p)\frac{h_2(0) - q}{h_2(0) - h_1(0)} + L_2(p)\frac{q - h_1(0)}{h_2(0)-h_1(0)} \right),
\end{equation*}
where
\begin{equation*}
L_i(p) = \begin{cases} \displaystyle{\frac{\dot h_i(t) + 1}{\dot h_i(0) + 1}p + \frac{\dot h_i(t) - \dot h_i(0)}{\dot h_i(0) + 1}} &\mbox{if } p \in [-1, \dot h_i(0)], \\ 
\displaystyle{\frac{1 - \dot h_i(t)}{1 - \dot h_i(0)}p + \frac{\dot h_i(t) - \dot h_i(0)}{1 - \dot h_i(0)}} & \mbox{if } p \in (\dot h_i(0), 1], \end{cases}\quad i = 1,2.
\end{equation*}
If we define $h$ in such a manner, then $m(q,p) = \pi_2 h(T, \pi_2 h^{-1}(0,q,p)) = \mathrm{id}_{W_0}$. Therefore, we obtain $\Lambda(m) = \chi(W_0)$, $\Lambda(m|_{W_0^{--}}) = \chi(W_0^{--})$.

\begin{figure}[!h]
\centering
\def\svgwidth{400px}
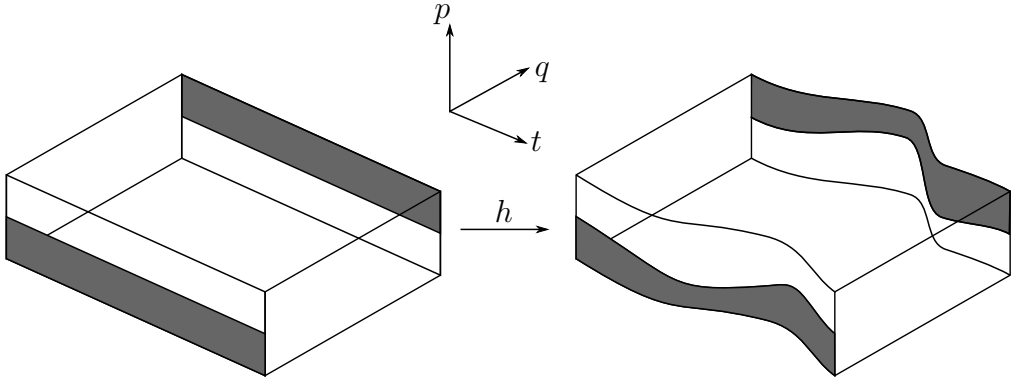
\caption{Qualitative representation of the periodic segment $W$ and a corresponding homeomorphism $h$. $[0,T] \times W^{--}_0$ and $W^{--}$ are in grey.}
\end{figure}

In our case $W_0$ is homeomorphic to a two-dimensional disk and $W_0^{--}$ is homeomorphic to two disjoint closed line segments. From the above, we obtain
\begin{equation*}
\Lambda(m) - \Lambda(m|_{W_0^{--}}) = \chi(W_0) - \chi(W_0^{--}) = -1,
\end{equation*}
and we can apply~(\ref{main-th}).
\end{proof}

\subsection{Applications}
\begin{example} Let us consider a forced relativistic pendulum in a gravitational field. Assume that its motion is described by the following equation
\begin{equation}
\label{exam1}
\frac{d}{dt}\frac{\dot \varphi}{\sqrt{1 - \dot \varphi^2}} = f(t, \varphi, \dot \varphi) + \alpha\sin\varphi,
\end{equation}
where $\alpha > 0$, and the function $f \in C^1(\mathbb{R}/T\mathbb{Z} \times \mathbb{R}/2\pi\mathbb{Z} \times \mathbb{R}, \mathbb{R})$, $T>0$ satisfies the following conditions
\begin{equation}
\label{ex1cond}
f(t, -\frac{\pi}{2}, 0) < \alpha, \quad f(t, \frac{\pi}{2}, 0) > -\alpha, \quad\mbox{for all}\quad t \in [0,T].
\end{equation}
Then there exists at least one periodic solution of~(\ref{exam1}). Moreover, along  this periodic solution the pendulum never approaches the horizontal position, i.e. it never falls. The last can be considered as an example of counter-intuitive behaviour since the pendulum is moving in the gravitational field and the external force $f$ can be arbitrarily large.

Let $h_1 = -\pi/2$ and $h_2 = \pi/2$, then, taking into account~(\ref{ex1cond}), we have
\begin{equation*}
\begin{aligned}
&f(t,h_1,\dot h_1) + \alpha\sin h_1 = f(t, -\frac{\pi}{2},0) - \alpha < 0,\\
&f(t,h_2,\dot h_2) + \alpha\sin h_2 = f(t, \frac{\pi}{2},0) + \alpha > 0,
\end{aligned}
\end{equation*}
and theorem~(\ref{ourres}) can be applied.
\end{example}
\begin{remark}
One can easily show that there also exists at least one periodic solution in the above system in the presence of viscous friction. Indeed, if for a given $\gamma \in \mathbb{R}$ we consider the modified equation
\begin{equation}
\label{exam1mod}
\frac{d}{dt}\frac{\dot \varphi}{\sqrt{1 - \dot \varphi^2}} = f(t, \varphi, \dot \varphi) - \gamma \dot\varphi + \alpha\sin\varphi,
\end{equation}
then conditions~(\ref{maincond}) are still satisfied for $h_1 = -\pi/2$ and $h_2 = \pi/2$.
\end{remark}
\begin{example}
Previous example can be generalized to the case of a forced relativistic particle moving in a gravitational field on a  curve. Suppose that $q \in \mathbb{R}$ is a natural parameter on a connected curve in $\mathbb{R}^2$. Then the equation of motion of this system is similar to~(\ref{exam1mod})
\begin{equation}
\label{exam2}
\frac{d}{dt}\frac{\dot q}{\sqrt{1 - \dot q^2}} = f(t, q, \dot q) - \gamma \dot q - \alpha y'(q),
\end{equation}
where $\gamma \in \mathbb{R}$ is a given parameter describing friction-like interaction (possibly zero), $\alpha >0$, and $y \in C^2(\mathbb{R}, \mathbb{R})$, $f \in C^1(\mathbb{R}/T\mathbb{Z} \times \mathbb{R} \times \mathbb{R}, \mathbb{R})$, $T>0$.
Similarly to the previous case, we assume that there are $q_1, q_2 \in \mathbb{R}$, $q_1 < q_2$ such that $y'(q_1) = 1$, $y'(q_2) = -1$, and $y(q) > 0$ for $q \in (q_1, q_2)$.  satisfies the following conditions
\begin{equation*}
f(t, q_1, 0) < \alpha, \quad f(t, q_2, 0) > -\alpha, \quad\mbox{for all}\quad t \in [0,T].
\end{equation*} 
Then there exists at least one $T$-periodic solution of~(\ref{exam2}) such that $q(t) \in (q_1, q_2)$ for all $t \in \mathbb{R}$. In particular, the point never approaches $y = 0$ along the obtained periodic solution, i.e. it never falls on the considered horizontal line. The proof in this case is similar to the case of a relativistic pendulum.
\end{example}
\begin{example}
Consider a relativistic pendulum with viscous friction (possibly zero) in a parallel force field. We assume that this external field is rotating and changing its absolute value in accordance with the prescribed laws $\psi \in C^2(\mathbb{R}/T\mathbb{Z}, \mathbb{R})$, $f \in C^1(\mathbb{R}/T\mathbb{Z}, \mathbb{R})$, $T > 0$.
\begin{figure}[!h]
\centering
\def\svgwidth{320px}
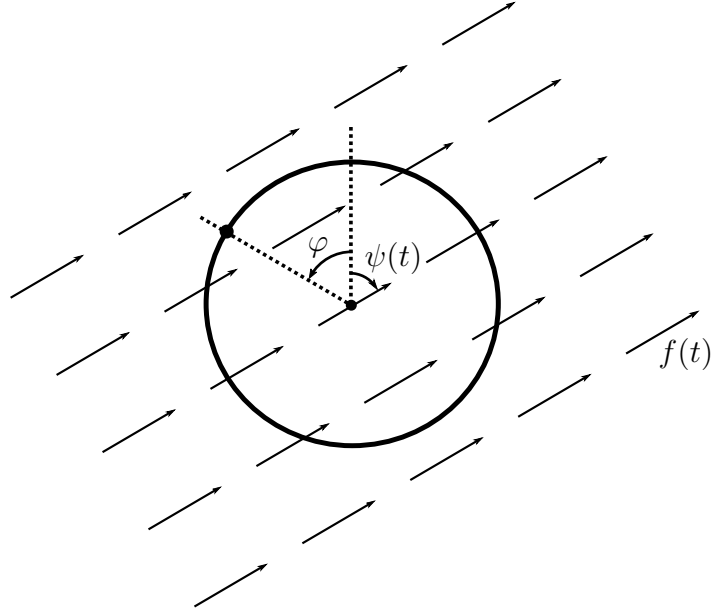
\caption{Relativistic pendulum in the rotating force field}
\end{figure}

\noindent The equation of motion is as follows
\begin{equation}
\label{exam3}
\frac{d}{dt}\frac{\dot \varphi}{\sqrt{1 - \dot \varphi^2}} = f(t)(\cos \varphi \sin\psi(t) -\sin \varphi \cos\psi(t)) - \gamma\dot\varphi,
\end{equation}
where $\gamma \in \mathbb{R}$. Assume that $|\dot\psi(t)| < 1$ and $f(t) > |\ddot\psi|(1 - \dot\psi^2)^{-3/2} + |\gamma\dot\psi|$ for all $t \in \mathbb{R}/T\mathbb{Z}$. Then there exists a $T$-periodic solution $\varphi$ such that $\varphi(t) \in (\psi(t) + \pi/2, \psi(t) + 3\pi/2)$, i.e. the particle follows the rotation of the force field, in particular, it performs the same number of revolutions around the pivot point of the pendulum as the force field.

To prove this result one can consider functions $h_1 = \psi + \pi/2$, $h_2 = \psi + 3\pi/2$ and apply~(\ref{ourres}) to~(\ref{exam3})
\begin{equation*}
\begin{aligned}
&(1 - \dot h_1^2)^{3/2}(f-\gamma\dot h_1)(\cos h_1 \sin\psi -\sin h_1 \cos\psi) - \ddot h_1 = -(1 - \dot\psi^2)^{3/2}(f+\gamma\dot\psi) - \ddot\psi < 0,\\
&(1 - \dot h_2^2)^{3/2}(f-\gamma\dot h_2)(\cos h_2 \sin\psi -\sin h_2 \cos\psi) - \ddot h_2 = (1 - \dot\psi^2)^{3/2}(f-\gamma\dot\psi) - \ddot\psi > 0.
\end{aligned}
\end{equation*}
\end{example}
\section{Conclusion}
\label{sec3}


We believe that the presented result can be useful in different applications in systems describing relativistic approximation of a constrained massive particle moving under the action of an external force which is possibly depends on time and the generalized coordinate and velocity of the particle. Theorem~(\ref{ourres}), if can be applied, allows one not only to prove the existence of a periodic solution, but also estimate it. At the same time, we have to mention that this constructive approach can be considered also as a shortcoming of the method comparing, for instance, to results \cite{torres2008periodic}, \cite{brezis2010periodic}, and~\cite{bereanu2012existence}, which can be used in a straightforward manner.




\bibliographystyle{elsarticle-num}


\bibliography{sample}

\end{document}